\documentclass[12pt]{article}

\usepackage[utf8]{inputenc}
\usepackage{xcolor}
\usepackage{hyperref}
\usepackage{amssymb,amsmath,amsfonts,graphicx,color,multicol,amsthm}
\usepackage{subcaption}
\usepackage[hmargin=.7in,vmargin=1.5in]{geometry}
\usepackage[T1]{fontenc}

\newtheorem{thm}{Theorem}[section]
\newtheorem{lem}[thm]{Lemma}
\newtheorem{cor}[thm]{Corollary}

\newtheorem{dfn}[thm]{Definition}
\newtheorem{rem}[thm]{Remark}

\begin{document}

\title{Linear multistep methods with repeated  global Richardson extrapolation}

\author{I.~Fekete\thanks{{\texttt{imre.fekete@ttk.elte.hu}}, corresponding author, Department of Network and Data Science, Central European University, Quellenstraße 51, 1100 Vienna, Austria and Department of Applied Analysis and Computational Mathematics, ELTE E\"otv\"os Lor\'and University, Pázmány P. s. 1/c, H-1117 Budapest, Hungary}, L.~L\'oczi\thanks{{\texttt{LLoczi@inf.elte.hu}}, Department of Numerical Analysis, ELTE E\"otv\"os Lor\'and University, P\'azm\'any P.~s.~1/c, H-1117 Budapest, Hungary, and Department of Analysis and Operations Research, BME Budapest University of Technology and Economics}}
\date{\today}

\maketitle

\begin{abstract}
In this work, we further investigate the application of the well-known Richardson extrapolation (RE) technique to accelerate the convergence of sequences resulting from linear multistep methods (LMMs) for numerically solving initial-value problems of systems of ordinary differential equations. By extending the ideas of our recent work on global
Richardson extrapolation,  
we now utilize some advanced versions of RE in the form of repeated RE (RRE). 
Assume that the underlying LMM---the base method---has order $p$ and RE is applied $\ell$ times. Then, we prove that the accelerated sequence has convergence order $p+\ell$. 
The version we present here is global RE (GRE, also known as passive RE), since the terms of the linear combinations are calculated independently. 
Thus, the resulting higher-order LMM-RGRE methods can be implemented in a parallel fashion and existing LMM codes can directly be used without any modification. We also investigate how the linear stability properties of the base method (e.g., $A$- or $A(\alpha)$-stability) are preserved by the LMM-RGRE methods. 
\end{abstract}


\section{Introduction}

Let us consider the initial-value problem
\begin{equation}\label{ODE}
y'(t)=f(t,y(t)), \quad y(t_0)=y_0,
\end{equation}
where $f:\mathbb{R}\times\mathbb{R}^m\to\mathbb{R}^m$ ($m\in\mathbb{N}^+$) is a given sufficiently smooth function. To approximate its unique solution $y$ on an interval $[t_0,t_{\text{final}}]$, one often applies a one-step or a linear multistep time-discretization method, resulting in a numerical sequence.

Classical Richardson extrapolation (RE) \cite{richardson1911, richardson1927} is a  technique to accelerate the convergence of numerical sequences depending on a small parameter by eliminating the lowest order error term(s) from the corresponding asymptotic expansion. When solving  \eqref{ODE} numerically, the 
parameter in RE can be chosen as the discretization step size $h>0$. The application of RE to sequences generated by one-step (e.g., Runge--Kutta) methods is described, for example, in \cite{butcher, hairernorsettwanner}. In \cite{zlatev}, global (also known as passive) and local (or active) versions of RE are implemented with Runge--Kutta sequences. These combined methods find application in many areas (see, e.g., \cite{falgout2021, zlatev2022}).

When carrying out global RE (GRE), one considers a suitable linear combination of two approximations, one generated on a coarser grid and one on a finer grid, to obtain a better approximation of the solution $y$ of \eqref{ODE}. This extrapolation is called global, because the sequences on the two grids are computed independently and their linear combination is formed only in the last step. Taking this idea further, one can consider several approximations on finer and finer grids to improve convergence even more. Such a procedure is called repeated global  Richardson extrapolation (RGRE). To describe RGRE, let us fix a basic step size (or grid length) $h$, some $\ell\in\mathbb{N}^+$ (denoting the number of RE applications), and
 a strictly increasing positive integer step-number sequence $1=n_1<n_2<\ldots<n_{\ell+1}$, then  consider the $\ell+1$ nested grids with grid lengths 
\[\frac{h}{n_1}, \frac{h}{n_2}, \ldots, \frac{h}{n_{\ell+1}},\] together with 
 the corresponding time-discretization sequences 
 \[y_{n_1\cdot n}\left(\frac{h}{n_1}\right), y_{n_2\cdot n}\left(\frac{h}{n_2}\right), \ldots,  y_{n_{\ell+1}\cdot n}\left(\frac{h}{n_{\ell+1}}\right).\] 
Let us denote by $r_n^{[\ell]}(h)$ the appropriately chosen linear combination of the above approximations (to be explained soon). For example, when the sequence $y_n(h)$ is generated by a one-step method of order $p$ applied to \eqref{ODE},  classical RE is recovered by choosing 
$\ell=1$, $(n_1,n_2)=(1,2)$ and 
\begin{equation}\label{egyszeresRE}
r_n^{[1]}(h):=\frac{2^p \cdot y_{2n}\left(\frac{h}{2}\right)- y_n(h)}{2^p-1}.
\end{equation}
In \cite{ZDFGH}, the authors present the formulae $r_n^{[\ell]}(h)$  for several larger values of $\ell$ corresponding to the step-number sequence $n_j=2^{j-1}$ ($j\ge 1$); here we just reproduce two more cases:


\begin{equation}\label{ketszeresRE}
r_n^{[2]}(h):=\frac{2^{2p+1}\cdot y_{4n}\left(\frac{h}{4}\right)-3\cdot 2^p\cdot y_{2n}\left(\frac{h}{2}\right)+ y_n(h)}{\left(2^p-1\right) \left(2^{p+1}-1\right)}
\end{equation}

\begin{equation}\label{haromszorosRE}
r_n^{[3]}(h):=\frac{2^{3p+3}\cdot y_{8n}\left(\frac{h}{8}\right)-7\cdot 2^{2p+1}\cdot y_{4n}\left(\frac{h}{4}\right)+7\cdot 2^{p}\cdot y_{2n}\left(\frac{h}{2}\right)- y_n(h)}{\left(2^p-1\right) \left(2^{p+1}-1\right)\left(2^{p+2}-1\right)}.
\end{equation}

\begin{rem}
    Note the slight difference in the  terminology: in \cite{ZDFGH}, $q$ denotes the number of RE repetitions, while here $\ell$ denotes the number of RE applications. In other words,  $q=\ell-1$, and classical RE corresponds to $\ell=1$ or $q=0$. 
\end{rem}

\begin{rem}\label{rem:ismeteltRE} 
In \cite{BH}, the authors describe multiple RE, which is 
another advanced version of Richardson extrapolation besides repeated RE. Based on \eqref{egyszeresRE}, multiple RE is defined as
\[
r_n^{\mathrm{multiple}}(h):=\frac{2^{p+1} \cdot r_{2n}^{[1]}\left(\frac{h}{2}\right)- r_n^{[1]}(h)}{2^{p+1}-1}.
\]
This formula is identical to $r_n^{[2]}(h)$ in \eqref{ketszeresRE}. (The context in \cite{BH} is, however, different, since they deal with \emph{local} RE applied to Runge--Kutta methods as underlying methods.)
\end{rem}

Let us briefly recall the construction of  formulae \eqref{egyszeresRE}--\eqref{haromszorosRE}. As mentioned earlier, extrapolation techniques are based on the existence of asymptotic expansions of the global discretization error with respect to the parameter $h$. 
One such classical theorem for one-step methods  (of order $p$) is due to Gragg (1964); for a modern treatment, see, e.g., \cite[Section II.8, Theorem 8.1]{hairernorsettwanner}. Assuming that $f$ in \eqref{ODE} is smooth enough, let us display, say, the first three terms of an asymptotic expansion. The theorem  then guarantees the existence of some functions $\mathbf{e}_{p}$, $\mathbf{e}_{p+1}$,$\mathbf{e}_{p+2}$ and $\mathbf{E}$ such that 
\begin{equation}\label{globerror}y_n(h)-y(t^*)=\mathbf{e}_p(t^*)\cdot h^p+
\mathbf{e}_{p+1}(t^*)\cdot h^{p+1}+\mathbf{e}_{p+2}(t^*)\cdot h^{p+2}
+\mathbf{E}(t^*,h)h^{p+3},
\end{equation}
where, given any  $\delta>0$ small enough, the function $\mathbf{E}$ is uniformly bounded for any $h\in[0,h_0]$ and any $t^*\in [t_0+\delta,t_{\text{final}}]$. (Clearly, for fixed $t^*$, we can choose $\delta$ such that $t^*\in[t_0+\delta,t_{\text{final}}]$.) The linear combinations in the definition of the sequences \eqref{egyszeresRE}--\eqref{haromszorosRE} have been set up such that the first, first two, and the first three terms, respectively, on the right-hand side of \eqref{globerror} are eliminated. More precisely,  to construct the coefficients of $\eqref{ketszeresRE}$, for example, one considers the expression 
\[
\gamma_1\cdot(y_n(h)-y(t^*))+\gamma_2\cdot\left(y_{2n}\left(\frac{h}{2}\right)-y(t^*)\right)+\gamma_3\cdot\left(y_{4n}\left(\frac{h}{4}\right)-y(t^*)\right),
\]
then applies \eqref{globerror} with $h$, $h/2$ and $h/4$, and 
solves the linear system for $\gamma_{1,2,3}$ obtained by setting $\gamma_1+\gamma_2+\gamma_3$ equal to 1, and the coefficients of $h^p$ and $h^{p+1}$ equal to 0. In other words,
$\gamma_{1,2,3}$ is the solution of the system
\[
\left(
\begin{array}{ccc}
 1 & 1 & 1 \\
 1 & 2^{-p} & 4^{-p} \\
 1 & 2^{-p-1} & 4^{-p-1} \\
\end{array}
\right) 
\left(
\begin{array}{c}
 \gamma_1 \\
 \gamma_2 \\
 \gamma_3 \\
\end{array}
\right)=\left(
\begin{array}{c}
 1 \\
 0 \\
 0 \\
\end{array}
\right).
\]
The above linear system possesses a unique solution, because matrices of the form
\[
\left(
\begin{array}{cccc}
 a_1^{\rho _1} & a_2^{\rho _1} & \cdots  & a_n^{\rho _1} \\
 a_1^{\rho _2} & a_2^{\rho _2} & \cdots  & a_n^{\rho _2} \\
 \vdots  & \vdots  &    & \vdots  \\
 a_1^{\rho _n} & a_2^{\rho _n} & \cdots  & a_n^{\rho _n} \\
\end{array}
\right)\in\mathbb{R}^{n\times n}
\]
are always invertible, whenever the positive real bases $a_j>0$ ($1\le j\le n$) are distinct, and the real exponents $\rho_j\in\mathbb{R}$ ($1\le j\le n$) are also distinct. This is a classical result of Laguerre (1883), see \cite[Section II.9, Theorem 9.1]{hairernorsettwanner} and 
\cite[Section II.9, Exercise 6]{hairernorsettwanner}. 
\begin{rem}
This result affirmatively answers \cite[Conjecture 1]{ZDFGH}, guaranteeing the existence of formulae $r_n^{[\ell]}(h)$ for \emph{any} number of RE applications $\ell\in\mathbb{N}^+$, moreover, not only for the step-number sequence $(1,2,4,\ldots,2^{\ell+1})$, but in fact for any positive integer step-number sequence $1=n_1<n_2<\ldots<n_{\ell+1}$.
\end{rem}

\noindent By following the approach outlined above, we present two more examples, corresponding to the ``most economic'' step-number sequence $(n_1,n_2,n_3,n_4)=(1,2,3,4)$:

\[
r_n^{[2]}(h):=\frac{3^{p+1}\cdot y_{3n}\left(\frac{h}{3}\right)- 2^{p+2}\cdot y_{2n}\left(\frac{h}{2}\right)+ y_n(h)}{3^{p+1}-2^{p+2}+1}
\]

\[
r_n^{[3]}(h):=\frac{4^{p+2}\cdot y_{4n}\left(\frac{h}{4}\right)-3^{p+3}\cdot y_{3n}\left(\frac{h}{3}\right)+3\cdot 2^{p+2}\cdot y_{2n}\left(\frac{h}{2}\right)- y_n(h)}{4^{p+2}-3^{p+3}+3\cdot 2^{p+2}-1}.
\]

It is now natural to ask what happens if we enter the realm of linear multistep methods (LMMs) as underlying numerical discretizations for \eqref{ODE}. Suppose we approximate the unique solution $y$ of \eqref{ODE} on $[t_0,t_{\text{final}}]$ by applying a $k$-step LMM 
\begin{equation}\label{multistepdefn}
\sum_{j=0}^k \alpha_j y_{n+j}=\sum_{j=0}^k h \beta_j f_{n+j}
\end{equation}
of order $p\ge 1$ on a uniform grid $\{t_n\}$ with step size $h=t_{n+1}-t_n>0$,  $f_m:=f(t_m, y_m)$, and the numbers $\alpha_j\in\mathbb{R}$ and $\beta_j\in \mathbb{R}$ ($j=0, \ldots, k$) being the given method coefficients with $\alpha_k\ne 0$. 
The LMM \eqref{multistepdefn} generates a sequence $y_n(h)$ which is supposed to approximate the exact solution at $t_n$, that is, $y_n(h)\approx y(t_n)$. In this work, this LMM will also be referred to as the base method or the underlying method.

Results in this direction include, e.g.,  \cite{D1} or \cite{D2}. In \cite{D1}, a special class of LMMs is developed possessing a certain type of asymptotic expansion of the global discretization error so that extrapolation techniques become applicable to improve their convergence order. In \cite{D2}, extrapolation techniques based on one-step methods and some LMMs (including the explicit midpoint rule or the implicit trapezoidal rule) have been thoroughly analyzed. 

In our recent work \cite{feketeloczi}, we have shown that the sequence $r_n^{[1]}(h)$ given by \eqref{egyszeresRE} converges to the solution of \eqref{ODE} (with $f$ assumed to be sufficiently smooth) if the component sequences $y_{2n}$ and $y_n$ are generated by a LMM \eqref{multistepdefn} of order $p\ge 1$: the order of convergence of $r_n^{[1]}(h)$ is at least $p+1$. The linear stability properties of the resulting method (i.e., LMM with $\ell=1$ GRE application) are also analyzed. In the convergence proof, we have utilized \cite[Section 6.3.4, Theorem 6.3.5]{gautschi}, which proves the existence of the \textit{first term} of an asymptotic expansion of the global discretization error. For that theorem to hold, one assumes that the $k$ starting values of the sequence $\{y_0(h),y_1(h),\ldots,y_{k-1}(h)\}$ to initiate the LMM  are each ${\cal{O}}(h^{p+1})$-close to the corresponding exact solution values. There is no assumption on the LMM itself (apart from having order $p$).

In the present work, we extend these ideas and prove that if the base LMM \eqref{multistepdefn} (satisfying  assumptions \boxed{a2}--\boxed{a4} listed in Theorem \ref{convergencethm} in Section \ref{secconv}) is of order $p$, and global Richardson extrapolation is applied $\ell$ times (for any $\ell\ge 1$),  then the sequences $r_n^{[\ell]}(h)$ converge to the solution of \eqref{ODE} with order of convergence $p+\ell$. We can refer to such  procedure as LMM-RGRE, or, when the number of RE applications is made explicit, LMM-$\ell$GRE. In order to prove convergence of $r_n^{[\ell]}(h)$ for larger values of $\ell$, one needs to guarantee the existence of \textit{more terms} in an  
asymptotic expansion of the global error. Here, 
we utilize \cite[Section III.9, Theorem 9.1]{hairernorsettwanner}, which is a fairly general theorem about the existence of several terms of an  asymptotic expansion of the global error for strictly stable general linear methods (GLMs). Since LMMs are special cases of GLMs, Theorem 9.1 is applicable in the present situation. This explains the imposition of assumption \boxed{a4} on the LMMs---LMMs satisfying  \boxed{a4} are also known as strongly (or strictly) stable LMMs. On the other hand, it turns out that it is enough to assume a weaker closeness relation   \boxed{a3}, that is, the starting values to initiate the LMM are only ${\cal{O}}(h^{p})$-close to the corresponding exact solution values.

The structure of our paper is as follows. Section \ref{sec:notation} summarizes some notation. In Section \ref{secconv}, the improved convergence rate of LMM-RGREs is proved, see Theorem \ref{convergencethm}. Possible classes of base LMMs include Adams--Bashforth, Adams--Moulton or BDF methods. Linear stability of LMM-RGREs is investigated in Section \ref{seclinstab}. Finally,  numerical tests are presented in Section \ref{secnumerics} to demonstrate the expected convergence order and efficiency.

\subsection{Notation}\label{sec:notation}
We assume throughout this work that $0\in\mathbb{N}$. The Kronecker product of two matrices is denoted by $\otimes$ (for its definition and properties, see, e.g.,  \cite{matrixmathematics}). For a (complex) square matrix $A$, let $\mathfrak{M}_A$ denote its minimal polynomial (that is, the unique univariate polynomial with leading coefficient equal to 1 and having the least degree such that $\mathfrak{M}_A(A)$ is the zero matrix).
For a set $S\subset\mathbb{C}$ and for $a>0$, we define $a\, S:=\{a z:z\in S\}$.
For $2\le k\in\mathbb{N}$, AB$k$, AM$k$ and BDF$k$ denote, respectively, the $k$-step Adams--Bashforth, Adams--Moulton, and BDF methods.







\section{Convergence analysis}\label{secconv}

It is known---\cite[Section III.4]{hairernorsettwanner} and \cite[Section III.8, Example 8.2]{hairernorsettwanner}---that any LMM  can be interpreted as a GLM written as a one-step method in a higher dimensional space. To this end, given the $\alpha$-coefficients of \eqref{multistepdefn}, define the matrix 
\begin{equation}\label{Amatrixdef}
    A:=\left(
\begin{array}{ccccc}
 -\alpha_{k-1}/\alpha_k  & -\alpha_{k-2}/\alpha_k  & \ldots  & -\alpha_{1}/\alpha_k  & -\alpha_{0}/\alpha_k  \\
 1  & 0  & \ldots  & 0  & 0  \\
 0  & 1  & \ldots  & 0  & 0  \\
 0  & 0  & \ddots  & \vdots  & \vdots  \\
 0  & 0  & \ldots  & 1  & 0  \\
\end{array}
\right)\in\mathbb{R}^{k\times k}
\end{equation}
so that its lower left block is the $(k-1)\times(k-1)$ identity matrix, then consider the recursion
\begin{equation}\label{onestepreformulation}
    Y_{n+1}=(A\otimes I)Y_n+h\Phi(t_n,Y_n,h)\quad\quad (n\in\mathbb{N}),
\end{equation}
where $Y_n:=(y_{n+k-1},y_{n+k-2},\ldots, y_n)^\top$ ($n\in\mathbb{N}$), $\Phi$ is the increment function of the numerical method (defined by \cite[Section III.4, formula (4.8a)]{hairernorsettwanner}, its particular structure will not be relevant), 
$I\in\mathbb{R}^{m \times m}$ is the identity matrix (recall that $m$ is the dimension of \eqref{ODE}, hence $A\otimes I\in\mathbb{R}^{(m\cdot k)\times (m\cdot k)}$).\\ 

Let us first recall \cite[Section III.9, Theorem 9.1]{hairernorsettwanner}---in a simplified form to minimize technicalities but highlight its relevance in the present context. The matrix $S$ appearing in the theorem below is 
$S:=A\otimes I$.

\begin{thm}[Hairer \& Lubich, 1984]\label{hairerlubich} Assume that a given GLM is consistent of order $p\ge 1$ and 
satisfies assumptions A1--A3:
\begin{itemize}
\item[A1] the GLM is strictly stable, that is, 
\begin{enumerate}
    \item[(i)] the matrix $S$ is power bounded, in other words, $\sup_{n\in\mathbb{N}}\|S^n\|\in\mathbb{R}$;
    \item[(ii)] 1 is the only eigenvalue of $S$ with modulus one;
\end{enumerate}
\item[A2] the functions $\alpha_n$ and $\beta_n$ in the definition of the increment function $\Phi$ are polynomials whose coefficients satisfy a certain exponential decay property;
\item[A3] the starting procedure, the correct-value function, and the increment function are sufficiently smooth.
\end{itemize}
Then the global discretization error has an asymptotic expansion of the form
\[
\mathbf{e}_p(t^*)\cdot h^p+\mathbf{e}_{p+1}(t^*)\cdot h^{p+1}+\ldots+
\mathbf{e}_{N}(t^*)\cdot h^{N}+
\mathbf{E}(t^*,h)h^{N+1},
\]
where $\mathbf{E}$ is bounded
uniformly in $h\in[0, h_0]$ and for $t^*$ in compact intervals not containing $t_0$.
\end{thm}

It will be useful to recall also the following result \cite[Chapter 7, Theorem 2.2.5]{zdzislav}.

\begin{thm}\label{zdzislav}
A square matrix $S$ is power bounded if and only if 
\begin{itemize}
    \item[(ia)] each  zero of its minimal polynomial lies in the closed unit disk, and 
    \item[(ib)] (potential) zeros of its minimal polynomial on the unit circle have multiplicity 1.
\end{itemize}
\end{thm}

As a last preparatory step, we prove the following statement about Kronecker products.

\begin{lem}\label{kroneckerminimal}
Let $M\in\mathbb{C}^{k\times k}$ be any square matrix, and let $I\in\mathbb{R}^{m\times m}$ be the identity matrix ($k,m\in\mathbb{N}^+$). Then, the minimal polynomials of $M$ and $M\otimes I$ coincide.
\end{lem}
\begin{proof}
In the proof, we denote identity matrices of appropriate size by the same symbol $I$. Consider a Jordan canonical form of $M=TJT^{-1}$. Then, one easily checks that $M\otimes I=(T\otimes I)(J\otimes I)(T\otimes I)^{-1}$. We know that if two matrices are similar, then their minimal polynomials coincide. So to prove the lemma, it is sufficient to show that $\mathfrak{M}_J=\mathfrak{M}_{J\otimes I}$. 

Let $J_1,\ldots,J_n$ denote the blocks of the block diagonal matrix $J$. Then, the blocks of the block diagonal matrix $J\otimes I$ are $J_1\otimes I,\ldots,J_n\otimes I$. It is easily seen that 
\[
\mathfrak{M}_J=\text{lcm}\left(\mathfrak{M}_{J_1},\mathfrak{M}_{J_2},\ldots,\mathfrak{M}_{J_n}\right), 
\]
where $\text{lcm}$ denotes the least common multiple of the given polynomials with leading coefficient set to $1$. Similarly, we have
\[
\mathfrak{M}_{J\otimes I}=\text{lcm}\left(\mathfrak{M}_{J_1\otimes I},\ldots,\mathfrak{M}_{J_n\otimes I}\right).
\]
Finally, we claim that $\mathfrak{M}_{J_j}=\mathfrak{M}_{J_j\otimes I}$ for any block $J_j\in\mathbb{R}^{s_j\times s_j}$ ($j=1,2,\ldots,n$). Indeed, let 
$\lambda_j$ be the eigenvalue corresponding to $J_j$, and $N$ the appropriate sized (nilpotent) square matrix with $1$s in its (first) superdiagonal and $0$s everywhere else. Then, for both possible types of Jordan blocks, we have
\begin{itemize}
    \item $J_j=\lambda_j I\Longrightarrow\mathfrak{M}_{J_j}(x)=x-\lambda_j=\mathfrak{M}_{J_j\otimes I}(x)$, 
    \item  $J_j=\lambda_j I+N \Longrightarrow\mathfrak{M}_{J_j}(x)=(x-\lambda_j)^{s_j}=\mathfrak{M}_{J_j\otimes I}(x)$,
\end{itemize}
completing the proof.
\end{proof}

Now we can present our main result establishing the improved convergence rates of LMM-$\ell$GREs. The fixed grid point $t^*:=t_0+n h$ is part of all of the grids. As always in this context, any implied constant in an ${\cal{O}}$ symbol is independent of $n$ and $h$.

\begin{thm}\label{convergencethm} Choose some 
$\ell\in\mathbb{N}^+$ and a positive integer step-number sequence $1=n_1<n_2<\ldots<n_{\ell+1}$ determining the grid lengths $\frac{h}{n_1}, \frac{h}{n_2}, \ldots, \frac{h}{n_{\ell+1}}$. Assume that
\begin{itemize}
    \item[\boxed{a1}] the right-hand side $f$ of the initial-value-problem \eqref{ODE} is sufficiently smooth;
    \item[\boxed{a2}] the LMM \eqref{multistepdefn} has order $p\ge 1$;
    \item[\boxed{a3}] for any $h\in[0,h_0]$, the $k$ starting values of the sequence \[\{y_0(h),y_1(h),\ldots,y_{k-1}(h)\}\] to initiate the LMM are each ${\cal{O}}(h^{p})$-close to the corresponding exact solution values
    \[\{y(t_0),y(t_0+h),\ldots,y(t_0+(k-1)h)\}\] of the initial-value-problem \eqref{ODE};
    \item[\boxed{a4}] the eigenvalues of  $A$ in \eqref{Amatrixdef} lie in the closed unit disk of the complex plane, the only eigenvalue with modulus 1 is 1, and its algebraic multiplicity is also 1.
\end{itemize}
Then, for any fixed grid point $t^*\in [t_0,t_{\emph{final}}]$ we have
\[r_n^{[\ell]}(h)-y(t^*)= {\cal{O}}(h^{p+\ell}).\]
\end{thm}
\begin{proof}
It is enough to show that our assumptions \boxed{a1}--\boxed{a4} imply the assumptions of Theorem \ref{hairerlubich}.
\noindent{\textbf{Step 1.}} Assumption \boxed{a1} implies the smoothness assumption A3 of Theorem \ref{hairerlubich}.\\
\noindent{\textbf{Step 2.}} Assumption A2 of Theorem \ref{hairerlubich} holds automatically since in our case the increment function $\Phi$ does not depend on $n$.\\
\noindent{\textbf{Step 3.}} Assumptions \boxed{a2}--\boxed{a3} imply the consistency of order $p$ of the GLM. Indeed, assumption \boxed{a2} implies that the local error of the LMM is ${\cal{O}}(h^{p+1})$ in the sense of \cite[Section III.2, Definition 2.3]{hairernorsettwanner}. But this quantity is just the first component of the vector $Y(x_{n+1})-\hat{Y}_{n+1}$ in \cite[Section III.4, Lemma 4.3]{hairernorsettwanner} describing the local error of the one-step reformulation \eqref{onestepreformulation} of the LMM;  the remaining components of $Y(x_{n+1})-\hat{Y}_{n+1}$ are $0$. This means that the local error $d_{n+1}$ ($n\in\mathbb{N}$) of the  GLM in the sense of \cite[Section III.8, Definition 8.9]{hairernorsettwanner} is also ${\cal{O}}(h^{p+1})$. Since $d_0={\cal{O}}(h^{p})$  due to assumption \boxed{a3}, we get, by using \cite[Section III.8, Lemma 8.11]{hairernorsettwanner}, that the consistency order of the GLM is $p$. (In (8.15) of \cite[Section III.8, Lemma 8.11]{hairernorsettwanner}, the relation with the spectral projector $E\delta_p(x)=0$ now clearly holds, since $\delta_p(x)=0$ due to $d_{n+1}={\cal{O}}(h^{p+1})$.)\\
\noindent{\textbf{Step 4.}} We finally show that assumption A1 of Theorem \ref{hairerlubich} follows from \boxed{a4}. It is known  \cite[Chapter 7]{matrixmathematics} that the eigenvalues of a Kronecker product $S=A\otimes I$ are the product of the eigenvalues of its components, hence (ii) of A1 of Theorem \ref{hairerlubich} is verified.
To check (i) of A1 of Theorem \ref{hairerlubich}, we use the characterization of power boundedness given by Theorem \ref{zdzislav}. 
The minimal polynomial divides the characteristic polynomial, so  \boxed{a4} imples (ia). We also know from \boxed{a4} that $1$ is a simple zero of the characteristic polynomial of $A$, but from the above properties of the Kronecker product we see that its multiplicity in the \emph{characteristic} polynomial of $S$ is $m$. However, $1$ is also a simple zero of the minimal polynomial of $A$. Hence, Lemma \ref{kroneckerminimal} with $M:=A$  implies property (ib).  
\end{proof}
\begin{rem}
    The characteristic polynomial of $A$ is a non-zero constant multiple of  the first characteristic polynomial of the LMM. Hence, due to consistency of the LMM \eqref{multistepdefn}, 1 is always an eigenvalue of $A$.  
\end{rem}
\begin{rem} Power boundedness of the matrix $S=A\otimes I$ in the above proof of Theorem \ref{convergencethm}  also follows from \cite[Section III.4, Lemma 4.4]{hairernorsettwanner}, since now $A$ is power bounded. Hence, our Lemma \ref{kroneckerminimal} can be considered as an alternative (more algebraic) proof of this fact.
\end{rem}

Many LMMs satisfy the assumptions \boxed{a2} and \boxed{a4} of Theorem \ref{convergencethm}. In particular, all AB and AM methods satisfy \boxed{a4} (since, for any of these methods, the only non-zero eigenvalue of the corresponding matrix \eqref{Amatrixdef} is $1$ with algebraic multiplicity $1$), and one directly checks that each BDF$k$ method also satisfies \boxed{a4}. Therefore, we have the following convergence result. 
\begin{cor}\label{maincorollary}
According to Theorem \ref{convergencethm}, the convergence of the sequences obtained by applying  any of the LMMs
\begin{itemize}
    \item AB$k$ (with $k\ge 2$ steps),
    \item AM$k$ (with $k \ge 2$ steps), or 
    \item BDF$k$ (with $2\le k\le 6$ steps)
\end{itemize}   
to problem \eqref{ODE} can be accelerated by using  LMM-$\ell$GREs.
\end{cor}
These combined methods  will be referred to as AB$k$-RGRE, AM$k$-RGRE, and BDF$k$-RGRE  in general, and AB$k$-$\ell$GRE, AM$k$-$\ell$GRE, and BDF$k$-$\ell$GRE  in particular (where $\ell$ is the number of RE applications appearing in $r_n^{[\ell]}(h)$).

\section{Linear stability analysis}\label{seclinstab}

In this section, let LMM denote any of the base methods AB$k$ ($k\ge 2$), AM$k$ ($k\ge 2$), or BDF$k$ ($2\le k\le 6$), and let $\ell\in\mathbb{N}^+$ denote the number of RE applications.  
 The region of absolute stability of the base method is denoted by $\mathcal{S}_\text{LMM}$, while that  of the combined LMM-${\ell}$GRE method will be denoted by ${\mathcal{S}}_\text{RGRE}^{[\ell]}$. Let us apply the LMM-${\ell}$GRE method to the scalar linear test equation 
$y'(t)=\lambda y(t)$, $y(t_0)=y_0$
 with some $h>0$ and $\lambda\in\mathbb{C}$. Similarly to \cite[Section 2.3]{feketeloczi}, we define ${\mathcal{S}}_\text{RGRE}^{[\ell]}\subset\mathbb{C}$ as follows. 
 
 \begin{dfn}{Suppose we are given a positive integer step-number sequence $1=n_1<n_2<\ldots<n_{\ell+1}$ that determines the RGRE. The region of absolute stability ${\mathcal{S}}_\emph{RGRE}^{[\ell]}\subset\mathbb{C}$ of the combined LMM-${\ell}$GRE method is defined as the set of numbers $\mu:=h\lambda$ for which the sequence $n\mapsto r_n^{[\ell]}(h)$ is bounded for any choice of the starting values of any of its component sequences
$n\mapsto y_{n}\left({h}/{n_j}\right)$ ($j=1,\ldots,\ell+1$), but excluding the values of $\mu$ for which any of the leading coefficients $\alpha_k-\left({\mu}/{n_j}\right)\beta_k$ ($j=1,\ldots,\ell+1$) vanishes.}
\end{dfn}
 The proofs of the lemmas of \cite[Section 2.3]{feketeloczi} can be extended in a straightforward way to the present situation, so we only state the results.

\begin{lem}\label{stabregionlemma} We have the inclusions 
 \[
\bigcap_{j=1}^{\ell+1} \big({n_j\,\mathcal{S}}_\text{\emph{LMM}}\big)\subseteq
{\mathcal{S}}_\text{\emph{RGRE}}^{[\ell]}\subseteq {\mathcal{S}}_\text{\emph{LMM}}.
 \]
\end{lem}
\begin{lem}
For any $2\le k\le 6$ and $\ell\in\mathbb{N}^+$, the  $\text{BDF}k$-$\ell$GRE method has the same $A(\alpha)$-stability angle as that of the underlying $\text{BDF}k$-method.
\end{lem}

\begin{lem}\label{convexlemma} Assume that  ${\mathcal{S}}_\text{\emph{LMM}}$ is convex. Then ${\mathcal{S}}_\text{\emph{RGRE}}^{[\ell]}={\mathcal{S}}_\text{\emph{LMM}}$.
\end{lem}

\begin{rem} In \cite{feketeloczi}, we have shown that ${\mathcal{S}}_\text{\emph{LMM}}$ is  convex, for example, for the AB2 and AM2 methods, but not convex for the AB3 method.
\end{rem}

\section{Numerical tests}\label{secnumerics}

In the context of Richardson extrapolation, taking into account standard implementation practices, we consider the step-number sequence $\left(n_1,n_2,\ldots,n_{\ell+1}\right)=(1,2,\ldots,2^{\ell+1})$ and verify the corresponding expected order of convergence of LMM-RGREs on three benchmark problems. 

We have chosen a Dahlquist test problem
\begin{align}
    y'(t)&=-5y(t)\label{eq:Dahlquist}
\end{align}
for $t\in[0,1]$ with initial condition $y(0)=1$, a Lotka--Volterra system
\begin{align}
    y'_1(t)&=0.1y_1(t)-0.3y_1(t)y_2(t),\label{eq:LotkaVolterra1}\\ 
    y'_2(t)&=0.5(y_1(t)-1)y_2(t) \label{eq:LotkaVolterra2}
\end{align}
for $t\in[0,62]$ with initial condition $y(0)=(1,1)^\top$, and a mildly stiff {van der Pol} equation 
\begin{align}
    y'_1(t)&=y_2(t), \label{eq:vanderPol1}\\ 
    y'_2(t)&=2(1-y_1^2(t))y_2(t)-y_1(t) \label{eq:vanderPol2}
\end{align}
for $t\in[0,20]$ with initial condition $y(0)=(2,0)^\top$.

As base LMMs, we considered the $2^\text{nd}$- and $3^\text{rd}$-order AB, AM, and BDF methods. For starting methods, we chose the $2^\text{nd}$- and $3^\text{rd}$-order Ralston methods, having minimum error bounds \cite{ralston1962}. As it is usual, the AM methods were implemented in predictor-corrector style. For the nonlinear algebraic equations arising in connection with implicit LMMs, we use MATLAB's \texttt{fsolve} command. When the goal is to achieve high convergence order or when we apply RE several times, one should change to a multi-precision environment due to the double-precision limitation of MATLAB and to the limitations of \texttt{fsolve}. The fine-grid solutions obtained by a $6^\text{th}$-order Runge--Kutta method with $2^{16}$ grid points are used as a reference solution to measure the global error in maximum norm and to estimate the order of convergence of LMM-RGREs \cite[Appendix A]{leveque}. 
Tables \ref{tab:Dahlquist_Lotka_Volterra_RRE_slope_4}, \ref{tab:Dahlquist_RRE_slope_5} and \ref{tab:Lotka_Volterra_RRE_slope_5}, and Figure \ref{fig:4_5_van_der_Pol} illustrate the expected $4^\text{th}$ or $5^\text{th}$-order convergence for these LMM-RGREs.


\begin{table}[ht!]
\begin{center}
\begin{tabular}{cccc||cccc}
\hline
$\text{AB}2$-2GRE & $\text{AM}2$-2GRE & $\text{BDF}2$-2GRE & $N$ & $\text{AB}2$-2GRE & $\text{AM}2$-2GRE & $\text{BDF}2$-2GRE & $N$\\ 
 \hline
 3.9437 & 4.0923 & 4.4613 & 128 & 3.9707  & 4.3757 & 3.8136 & 1024\\
 3.9833 & 4.0469 & 4.2500 & 256 & 3.9883  & 4.2391 & 3.9265 & 2048\\
 3.9942 & 4.0237 & 4.1333 & 512 &3.9951  & 4.1376 & 3.9679 & 4096\\
 3.9977 & 4.0119 & 4.0342 & 1024 &3.9983  & 4.0928 & 3.9908 & 8192\\
 \hline
 \multicolumn{4}{c}{Dahlquist problem \eqref{eq:Dahlquist}} & \multicolumn{4}{c}{Lotka--Volterra system \eqref{eq:LotkaVolterra1}-\eqref{eq:LotkaVolterra2}} 
\end{tabular}

\caption{The estimated order of convergence for different $4^\text{th}$-order LMM-RGREs, where $N$ denotes the number of grid points of the finer mesh}
\label{tab:Dahlquist_Lotka_Volterra_RRE_slope_4}
\end{center}
\end{table}

\begin{table}[ht!]
\begin{center}
\begin{tabular}{ccccccc}
\hline
$\text{AB}3$-2GRE & $\text{AM}3$-2GRE & $\text{BDF}3$-2GRE & 
$\text{AB}2$-3GRE & $\text{AM}2$-3GRE & $\text{BDF}2$-3GRE & $N$ \\
 \hline
 5.0121 & 5.0319 & 5.2081 & 4.8145 & 5.0564 & 5.1449 & 512 \\
 \hline
\end{tabular}
\caption{The estimated order of convergence for the Dahlquist problem \eqref{eq:Dahlquist} for different $5^\text{th}$-order LMM-RGREs, where $N$ denotes the number of grid points of the finer mesh}\label{tab:Dahlquist_RRE_slope_5}
\end{center}
\end{table}   

\begin{table}[ht!]
\begin{center}
\begin{tabular}{ ccccccc}
\hline
$\text{AB}3$-2GRE & $\text{AM}3$-2GRE & $\text{BDF}3$-2GRE & $\text{AB}2$-3GRE & $\text{AM}2$-3GRE & $\text{BDF}2$-3GRE & $N$\\
 \hline
 4.4933  & 4.8252 & 5.4540 & 6.1243 & 4.2922 & 4.0779 & 512\\
 5.0397  & 4.9769 & 5.4343 & 5.4060 & 4.8533 & 4.7738 & 1024 \\
 5.0747  & 4.9996 & 5.2692 & 5.2996 & 4.9615 & 4.9942 & 2048\\
 4.9983  & 4.9509 & 5.2278 & 5.2431 & 4.9864 & 4.9914 & 4096\\
 \hline
\end{tabular}
\caption{The estimated order of convergence for the Lotka--Volterra system \eqref{eq:LotkaVolterra1}-\eqref{eq:LotkaVolterra2} for different $5^\text{th}$-order LMM-RGREs, where $N$ denotes the number of grid points of the finer mesh}
\label{tab:Lotka_Volterra_RRE_slope_5}
\end{center}
\end{table}


\begin{figure}[ht!]
\includegraphics[width=1.0\textwidth]{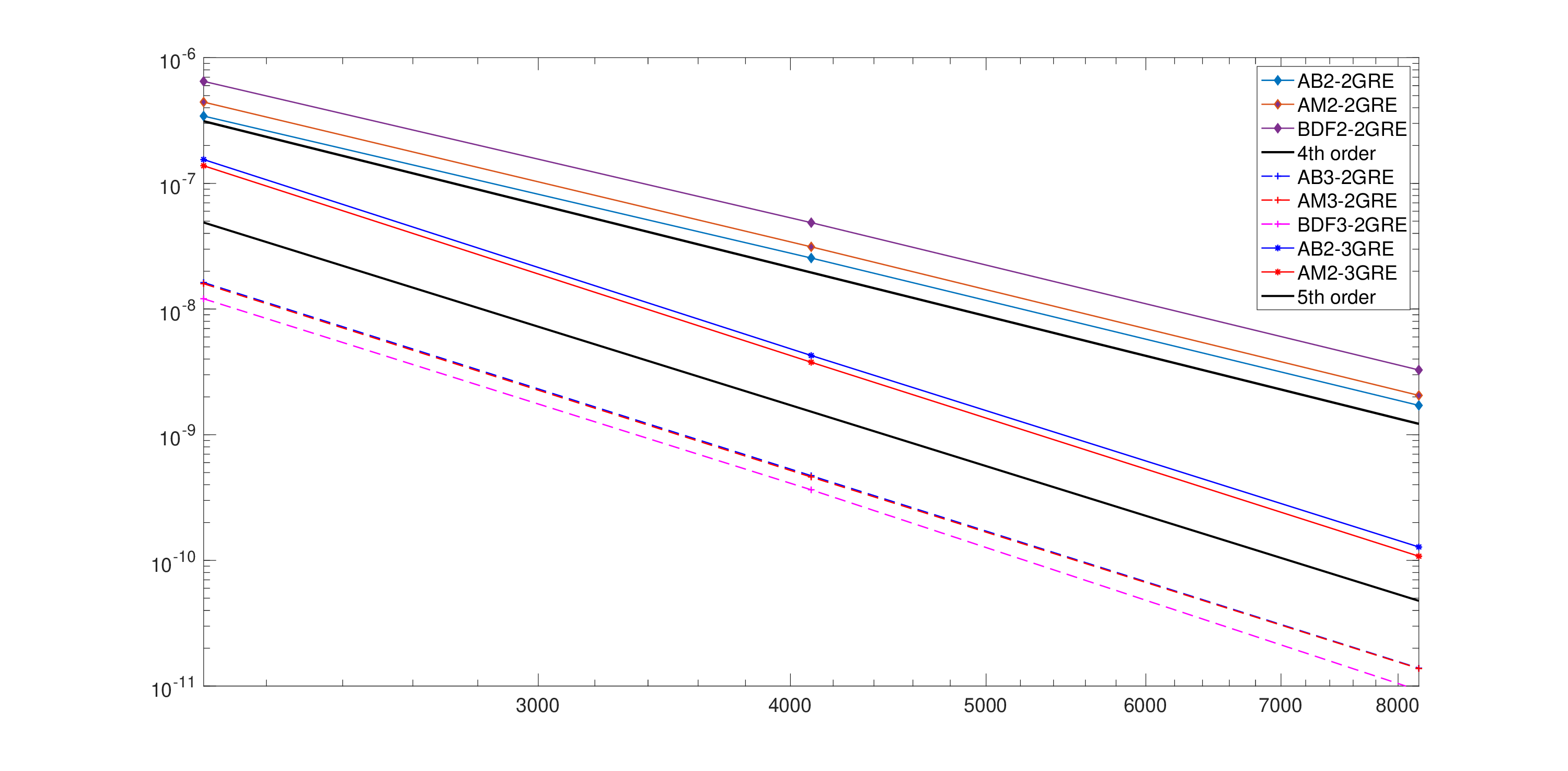}
\caption{Number of grid points versus the global error in maximum norm for the van der Pol equation \eqref{eq:vanderPol1}-\eqref{eq:vanderPol2}
for $4^\text{th}$-order and $5^\text{th}$-order LMM-RGREs}\label{fig:4_5_van_der_Pol}
\end{figure}

It is natural to ask whether the LMM-RGREs can be computationally more efficient than the corresponding underlying methods. To test this, we have chosen the mildly stiff van der Pol system \eqref{eq:vanderPol1}-\eqref{eq:vanderPol2} being the most challenging problem in this section. As shown in Table \ref{tab:van_der_Pol_running_times}, the LMM-RGREs  reach the given tolerance value in significantly less time than their underlying LMM counterparts.

\begin{table}[ht!]
\begin{center}
\begin{tabular}{c|cc|cc}
\hline
Running time (sec) & $\text{AB}2$& $\text{AB}2$-2GRE & $\text{AM}2$& $\text{AM}2$-2GRE \\ 
 \hline
 Minimum & 0.0305 & 0.0030 &  0.0337 & 0.0072 \\
 Average & 0.0340 & 0.0037 & 0.0365 & 0.0082 \\
 \hline
\end{tabular}

\caption{The running times measured in seconds for the van der Pol system \eqref{eq:vanderPol1}-\eqref{eq:vanderPol2} for different $2^\text{nd}$-order LMMs and the corresponding $4^\text{th}$-order LMM-RGREs with tolerance $10^{-6}$, where the minimum and average values were computed after $100$ simulations. In this case the reference solution has been computed using $2^{20}$ grid points.}
\label{tab:van_der_Pol_running_times}
\end{center}
\end{table}


\newpage

\section{Conclusions}

Based on the results of Sections \ref{secconv}--\ref{secnumerics}, we can draw the following conclusions.

\begin{itemize}
    \item  To implement a  LMM-RGRE method,  existing LMM codes can directly be used without any modification.
    

    \item  Although our recent work \cite{feketeloczi} can be considered as the $\ell=1$ special case of Theorem \ref{convergencethm} of this paper, still, the assumptions of Theorem \ref{convergencethm} on the closeness of the initial values are relaxed: instead of ${\cal{O}}(h^{p+1})$-closeness, it is sufficient to assume ${\cal{O}}(h^{p})$-closeness for the starting values of the LMM.

     \item The proof of Theorem \ref{convergencethm} is based on \cite[Section III.9]{hairernorsettwanner}. Since they establish the existence of 
    asymptotic expansions of the global error for (strictly stable) general linear methods (GLMs), one can clearly apply RGRE in that context as well, and accelerate the convergence of numerical sequences generated by such GLMs.

    \item Regarding linear stability, when the underlying LMM is a BDF method, for example, LMM-RGREs preserve the  $A(\alpha)$-stability angles. In particular, when $\ell=2$ and the base method is a BDF2 method, we obtain a $4^\text{th}$-order $A$-stable method. Another example is the BDF5 method with $\ell=2$, resulting  in a $7^\text{th}$-order method with $A(\alpha)$-stability angle $\approx 51.839^\circ$ (see also \cite[Table 1]{feketeloczi} for the $\ell=1$ case).

\item The computational cost of a LMM-RGRE increases for larger values of the RE applications. In general, due to the computations on finer grids, applying GRE $\ell$ times requires approximately $n_1+n_2+\ldots +n_{\ell+1}$ as much computation as the underlying LMM. However, this is compensated by the higher convergence order $p+\ell$. In particular, the LMM-RGREs shown in Table \ref{tab:van_der_Pol_running_times} are several times faster than their underlying LMM counterparts when the goal is to reach a given error tolerance.
    
\end{itemize}
 
The application of \textit{local} RE to LMMs is the subject of an ongoing study. We also plan to
systematically investigate the computational efficiency of various RE implementations.


    


\section*{Acknowledgement}


The authors would like to thank the anonymous reviewers for their careful comments that improved the structure and extended the scope of the paper.

I.~Fekete was supported by the J\'anos Bolyai Research Scholarship of the Hungarian Academy of Sciences, and also supported by the \'UNKP-22-5 New National Excellence Program of the Ministry for Culture and Innovation from the source of the National Research, Development and Innovation Fund. 

The research was supported by the Ministry of Innovation and Technology NRDI Office within the framework of the Artificial Intelligence National Laboratory Program RRF-2.3.1-21-2022-00004. 

\section*{Declaration}
\textbf{Conflict of interest} The authors declare that there is no conflict of interest.


\end{document}